\documentclass{article}

\usepackage{amsmath,amsthm,amssymb}
\usepackage{macros}
\newtheorem{theorem}{Theorem}

\newtheorem{assumption}[theorem]{Assumption}
\newtheorem{lemma}[theorem]{Lemma}
\newtheorem{corollary}[theorem]{Corollary}
\newtheorem{proposition}[theorem]{Proposition}

\usepackage{graphicx}
\graphicspath{{figures/}}
\let\epsilon\varepsilon

\usepackage{times}
\usepackage{graphicx} 
\usepackage{subfigure} 

\usepackage{natbib}

\usepackage{algorithm}
\usepackage{algorithmic}

\usepackage{hyperref}

\usepackage[accepted]{icml2015}

\icmltitlerunning{A General Analysis of the Convergence of ADMM}

\begin{document} 

\twocolumn[
\icmltitle{A General Analysis of the Convergence of ADMM}

\icmlauthor{Robert Nishihara}{rkn@eecs.berkeley.edu}
\icmlauthor{Laurent Lessard}{lessard@berkeley.edu}
\icmlauthor{Benjamin Recht}{brecht@eecs.berkeley.edu}
\icmlauthor{Andrew Packard}{apackard@berkeley.edu}
\icmlauthor{Michael I.~Jordan}{jordan@eecs.berkeley.edu}
\icmladdress{University of California,
            Berkeley, CA 94720 USA}

\icmlkeywords{ADMM, optimization, semidefinite programming, convergence rates}

\vskip 0.3in
]

\begin{abstract} 
We provide a new proof of the linear convergence of the alternating direction method of multipliers (ADMM) when one of the objective terms is strongly convex. 
Our proof is based on a framework for analyzing optimization algorithms introduced in \citet{lessard2014analysis}, reducing algorithm convergence to verifying the stability of a dynamical system. 
This approach generalizes a number of existing results and obviates any assumptions about specific choices of algorithm parameters. 
On a numerical example, we demonstrate that minimizing the derived bound on the convergence rate provides a practical approach to selecting algorithm parameters for particular ADMM instances. 
We complement our upper bound by constructing a nearly-matching lower bound on the worst-case rate of convergence. 
\end{abstract} 

\section{Introduction}
\label{sec:introduction}

The alternating direction method of multipliers (ADMM) seeks to solve the problem
\begin{equation} \label{eq:minimization_problem}
\begin{aligned}
  \text{minimize} & \quad f(x) + g(z) \\
  \text{subject to} & \quad Ax + Bz = c ,
\end{aligned}
\end{equation}
with variables~$x \in \mathbb R^p$ and~$z \in \mathbb R^q$ and constants~$A \in \mathbb R^{r \times p}$,~$B \in \mathbb R^{r \times q}$, and~$c \in \mathbb R^r$. 
ADMM was introduced in \citet{glowinski1975approximation} and \citet{gabay1976dual}. 
More recently, it has found applications in a variety of distributed settings such as model fitting, resource allocation, and classification. 
A partial list of examples includes \citet{bioucas2010alternating,wahlberg2012admm,bird2014optimizing,forero2010consensus,sedghi2014multi,li2014communication,wang2012online,zhang2012efficient,meshi2011alternating,burges2013large,aslan2013convex,forouzan2013linear,romera2013new,behmardi2014overlapping,zhang2014asynchronous}. 
See \citet{boyd2011distributed} for an overview. 

Part of the appeal of ADMM is the fact that, in many contexts, the algorithm updates lend themselves to parallel implementations. 
The algorithm is given in \algoref{alg:admm}. 
We refer to~$\rho>0$ as the step-size parameter. 

\begin{algorithm}
  \caption{Alternating Direction Method of Multipliers}
  \label{alg:admm}
\begin{algorithmic}[1]
  \STATE {\bfseries Input:} functions~$f$ and~$g$, matrices~$A$ and~$B$, vector~$c$, parameter~$\rho$
  \STATE Initialize~$x_0,z_0,u_0$
  \REPEAT
  \STATE $x_{k+1} = \argmin_x f(x) + \tfrac{\rho}{2} \|Ax + Bz_k - c + u_k\|^2$
  \STATE $z_{k+1} = \argmin_z g(z) + \tfrac{\rho}{2} \|Ax_{k+1} + Bz - c + u_k\|^2$
  \STATE $u_{k+1} = u_k + Ax_{k+1} + Bz_{k+1} - c$.
  \UNTIL{meet stopping criterion}
\end{algorithmic}
\end{algorithm}

A popular variant of \algoref{alg:admm} is over-relaxed ADMM, which introduces an additional parameter~$\alpha$ and replaces each instance of~$Ax_{k+1}$ in the~$z$ and~$u$ updates in \algoref{alg:admm} with
\begin{equation*}
  \alpha A x_{k+1} - (1-\alpha)(Bz_k-c) .
\end{equation*}
The parameter~$\alpha$ is typically chosen to lie in the interval~$(0,2]$, but we demonstrate in \secref{sec:selecting_hyperparameters} that a larger set of choices can lead to convergence. 
Over-relaxed ADMM is described in \algoref{alg:overrelaxed_admm}. 
When~$\alpha=1$, \algoref{alg:overrelaxed_admm} and \algoref{alg:admm} coincide. 
We will analyze \algoref{alg:overrelaxed_admm}. 

\begin{algorithm}
\caption{Over-Relaxed Alternating Direction Method of Multipliers}
\label{alg:overrelaxed_admm}
\begin{algorithmic}[1]
  \STATE {\bfseries Input:} functions~$f$ and~$g$, matrices~$A$ and~$B$, vector~$c$, parameters~$\rho$ and~$\alpha$
  \STATE Initialize~$x_0,z_0,u_0$
  \REPEAT
  \STATE $x_{k+1} = \argmin_x f(x) + \tfrac{\rho}{2} \|Ax + Bz_k - c + u_k\|^2$
  \STATE $z_{k+1} = \argmin_z g(z) + \tfrac{\rho}{2} \|\alpha Ax_{k+1} - (1-\alpha)Bz_k  + Bz - \alpha c + u_k\|^2$
  \STATE $u_{k+1} = u_k + \alpha Ax_{k+1} - (1-\alpha)Bz_k + Bz_{k+1} - \alpha c$
  \UNTIL{meet stopping criterion}
\end{algorithmic}
\end{algorithm}

The conventional wisdom that ADMM works well without any tuning \citep{boyd2011distributed}, for instance by setting~$\rho=1$, is often not borne out in practice. 
\algoref{alg:admm} can be challenging to tune, and \algoref{alg:overrelaxed_admm} is even harder. 
We use the machinery developed in this paper to make reasonable recommendations for setting~$\rho$ and~$\alpha$ when some information about~$f$ is available (\secref{sec:selecting_hyperparameters}). 

In this paper, we give an upper bound on the linear rate of convergence of \algoref{alg:overrelaxed_admm} for all~$\rho$ and~$\alpha$ (\thref{th:analytic_rate}), and we give a nearly-matching lower bound (\thref{th:lower_bound}). 

Importantly, we show that we can prove convergence rates for \algoref{alg:overrelaxed_admm} by numerically solving a~$4\times 4$ semidefinite program (\thref{th:lmi_rate}). 
When we change the parameters of \algoref{alg:overrelaxed_admm}, the semidefinite program changes. 
Whereas prior work requires a new proof of convergence for every change to the algorithm, our work automates that process. 

Our work builds on the integral quadratic constraint framework introduced in \citet{lessard2014analysis}, which uses ideas from robust control to analyze optimization algorithms that can be cast as discrete-time linear dynamical systems. 
Related ideas, in the context of feedback control, appear in \citet{corless1990guaranteed,d'alto2013incremental}. 
Our work provides a flexible framework for analyzing variants of \algoref{alg:admm}, including those like \algoref{alg:overrelaxed_admm} created by the introduction of additional parameters. 
In \secref{sec:related_work}, we compare our results to prior work.

\section{Preliminaries and Notation}
\label{sec:preliminaries}

Let~$\overline{\mathbb R}$ denote the extended real numbers~$\mathbb R \cup \{+\infty\}$. 
Suppose that~$f \colon \mathbb R^d \to \overline{\mathbb R}$ is convex and differentiable, and let~$\nabla f$ denote the gradient of~$f$. 
We say that~$f$ is strongly convex with parameter~$m > 0$ if for all~$x,y \in \mathbb R^d$, we have
\begin{equation*}
  f(x) \ge f(y) + \nabla f(y)^{\top}(x-y) + \tfrac{m}{2} \|x-y\|^2 .
\end{equation*}
When~$\nabla f$ is Lipschitz continuous with parameter~$L$, then
\begin{equation*}
  f(x) \le f(y) + \nabla f(y)^{\top}(x-y) + \tfrac{L}{2} \|x-y\|^2 .
\end{equation*}
For~$0 < m \le L < \infty$, let~$S_d(m,L)$ denote the set of differentiable convex functions~$f \colon \mathbb R^d \to \overline{\mathbb R}$ that are strongly convex with parameter~$m$ and whose gradients are Lipschitz continuous with parameter~$L$. 
We let~$S_d(0,\infty)$ denote the set of convex functions~$\mathbb R^d \to \overline{\mathbb R}$. 
In general, we let~$\partial f$ denote the subdifferential of~$f$. 
We denote the~$d$-dimensional identity matrix by~$I_d$ and the~$d$-dimensional zero matrix by~$0_d$. 
We will use the following results. 
\begin{lemma} \label{lem:sec_iqc}
  Suppose that~$f \in S_d(m,L)$, where~$0 < m \le L < \infty$. 
  Suppose that~$b_1 = \nabla f(a_1)$ and~$b_2 = \nabla f(a_2)$. 
  Then
  \begin{equation*}
    \begin{bmatrix} a_1-a_2 \\ b_1-b_2 \end{bmatrix}^{\top}
    \begin{bmatrix} -2mL I_d & (m+L)I_d \\ (m+L)I_d & -2I_d \end{bmatrix}
    \begin{bmatrix} a_1-a_2 \\ b_1-b_2 \end{bmatrix}
    \ge 0 .
  \end{equation*}
\end{lemma}
\begin{proof}
  The Lipschitz continuity of~$\nabla f$ implies the co-coercivity of~$\nabla f$, that is
\begin{equation*} \label{eq:co-coercivity}
  (a_1 - a_2)^{\top} (b_1-b_2) \ge \tfrac{1}{L} \| b_1 - b_2 \|^2 .
\end{equation*}
Note that~$f(x)-\frac{m}{2}\|x\|^2$ is convex and its gradient is Lipschitz continuous with parameter~$L-m$. 
Applying the co-coercivity condition to this function and rearranging gives
\begin{equation*}
  (m+L)(a_1 - a_2)^{\top} (b_1-b_2)  \ge mL \|a_1 - a_2\|^2 + \| b_1 - b_2 \|^2 ,
\end{equation*}
which can be put in matrix form to complete the proof. 
\end{proof}

\begin{lemma} \label{lem:sec_iqc2}
  Suppose that~$f \in S_d(0,\infty)$, and suppose that~$b_1 \in \partial f(a_1)$ and~$b_2 \in \partial f(a_2)$. 
  Then
  \begin{equation*}
    \begin{bmatrix} a_1-a_2 \\ b_1-b_2 \end{bmatrix}^{\top}
    \begin{bmatrix} 0_d & I_d \\ I_d & 0_d \end{bmatrix}
    \begin{bmatrix} a_1-a_2 \\ b_1-b_2 \end{bmatrix}
    \ge 0 .
  \end{equation*}
\end{lemma}
\lemref{lem:sec_iqc2} is simply the statement that the subdifferential of a convex function is a monotone operator. 

When~$M$ is a matrix, we use~$\kappa_M$ to denote the condition number of~$M$. 
For example,~$\kappa_A=\sigma_1(A)/\sigma_p(A)$, where~$\sigma_1(A)$ and~$\sigma_p(A)$ denote the largest and smallest singular values of the matrix~$A$. 
When~$f \in S_d(m,L)$, we let~$\kappa_f=\frac{L}{m}$ denote the condition number of~$f$. 
We denote the Kronecker product of matrices~$M$ and~$N$ by~$M\otimes N$.

\section{ADMM as a Dynamical System} \label{sec:admm_as_dynamical_system}

We group our assumptions together in \assref{ass:conditions}. 
\begin{assumption} \label{ass:conditions}
  We assume that~$f$ and~$g$ are convex, closed, and proper. 
  We assume that for some~$0 < m \le L < \infty$, we have~$f \in S_p(m,L)$ and~$g \in S_q(0,\infty)$. 
  We assume that~$A$ is invertible and that~$B$ has full column rank. 
\end{assumption}

The assumption that~$f$ and~$g$ are closed (their sublevel sets are closed) and proper (they neither take on the value~$-\infty$ nor are they uniformly equal to~$+\infty$) is standard.

We begin by casting over-relaxed ADMM as a discrete-time dynamical system with state sequence~$(\xi_k)$, input sequence~$(\nu_k)$, and output sequences~$(w_k^1)$ and~$(w_k^2)$ satisfying the recursions
\begin{subequations}
\label{eq:admm_dynamical_system_update_output} 
\begin{align} 
  \xi_{k+1} & = (\hat{A} \otimes I_r) \xi_k + (\hat{B} \otimes I_r) \nu_k \label{eq:admm_dynamical_system_update_output1} \\
  w_k^1 & = (\hat{C}^{1} \otimes I_r) \xi_k + (\hat{D}^{1} \otimes I_r) \nu_k \label{eq:admm_dynamical_system_update_output2} \\
  w_k^2 & = (\hat{C}^{2} \otimes I_r) \xi_k + (\hat{D}^{2} \otimes I_r) \nu_k  \label{eq:admm_dynamical_system_update_output3}
\end{align}
\end{subequations}
for particular matrices~$\hA$,~$\hB$,~$\hC^1$,~$\hD^1$,~$\hC^2$, and~$\hD^2$ (whose dimensions do not depend on any problem parameters). 

First define the functions~$\hat{f},\hat{g} \colon \mathbb R^r \to \overline{\mathbb R}$ via
\begin{equation} \label{eq:def_fhat_ghat}
\begin{aligned}
  \hat{f} & = (\rho^{-1}f)\circ A^{-1} \\
  \hat{g} & = (\rho^{-1}g)\circ B^{\dagger} + \mathbb I_{\image{B}} ,
\end{aligned}
\end{equation}
where~$B^{\dagger}$ is any left inverse of~$B$ and where~$\mathbb I_{\image{B}}$ is the~$\{0,\infty\}$-indicator function of the image of~$B$. 
We define~$\kappa = \kappa_f\kappa_A^2$ and to normalize we define
\begin{equation} \label{eq:normalization_defs}
  \hat{m}=\frac{m}{\sigma_1^2(A)}
  \quad 
  \hat{L}=\frac{L}{\sigma_p^2(A) } 
  \quad 
  \rho = (\hat{m}\hat{L})^{\frac12}\rho_0 .
\end{equation}
Note that under \assref{ass:conditions}, 
\begin{subequations} \label{eq:fhat_ghat}
\begin{align}
  \hat{f} & \in S_p(\rho_0^{-1}\kappa^{-\frac{1}{2}}, \rho_0^{-1}\kappa^{\frac{1}{2}}) \label{eq:fhat_ghata} \\
  \hat{g} & \in S_q(0, \infty) . \label{eq:fhat_ghatb}
\end{align}
\end{subequations}

To define the relevant sequences, let the sequences $(x_k)$,~$(z_k)$, and~$(u_k)$ be generated by \algoref{alg:overrelaxed_admm} with parameters~$\alpha$ and~$\rho$. 
Define the sequences~$(r_k)$ and~$(s_k)$ by~$r_k=Ax_k$ and~$s_k=Bz_k$ and the sequence~$(\xi_k)$ by
\begin{equation*}
  \xi_k=\begin{bmatrix} s_k \\ u_k \end{bmatrix} . 
\end{equation*}
We define the sequence~$(\nu_k)$ as in \propref{prop:dynamical_system_update}. 

\begin{proposition} \label{prop:dynamical_system_update}
  There exist sequences~$(\beta_k)$ and~$(\gamma_k)$ with~$\beta_k = \nabla \hat{f}(r_k)$ and~$\gamma_k \in \partial \hat{g}(s_k)$ such that when we define the sequence~$(\nu_k)$ by
\begin{equation*} 
  \nu_k=\begin{bmatrix} \beta_{k+1} \\ \gamma_{k+1} \end{bmatrix},
\end{equation*}
then~$(\xi_k)$ and~$(\nu_k)$ satisfy \eqref{eq:admm_dynamical_system_update_output1} with the matrices
\begin{equation} \label{eq:admm_matrices_update}
  \hat{A} = \begin{bmatrix} 1 & \alpha-1 \\ 0 & 0 \end{bmatrix}
\quad\quad
  \hat{B} = \begin{bmatrix} \alpha & -1 \\ 0 & -1 \end{bmatrix}  .
\end{equation} 
\end{proposition}
\begin{proof}
Using the fact that~$A$ has full rank, we rewrite the update rule for~$x$ from \algoref{alg:overrelaxed_admm} as
\begin{align*}
  x_{k+1} = A^{-1} \argmin_{r} f(A^{-1}r) + \tfrac{\rho}{2} \|r + s_k - c + u_k \|^2 .
\end{align*}
Multiplying through by~$A$, we can write
\begin{equation*}
  r_{k+1} = \argmin_{r} \hat{f}(r) + \tfrac{1}{2} \|r + s_k - c + u_k \|^2 .
\end{equation*}
This implies that
\begin{equation*}
  0 = \nabla \hat{f}(r_{k+1}) + r_{k+1} + s_k - c + u_k ,
\end{equation*}
and so
\begin{equation} \label{eq:r_update}
  r_{k+1} = -s_k - u_k + c - \beta_{k+1} ,
\end{equation}
where~$\beta_{k+1}=\nabla \hat{f}(r_{k+1})$. 
In the same spirit, we rewrite the update rule for~$z$ as
\begin{align*}
 s_{k+1} & = \argmin_{s} \hat{g}(s) \\
 & \quad\quad\quad + \tfrac{1}{2} \|\alpha r_{k+1} - (1-\alpha)s_k + s - \alpha c + u_k \|^2 .
\end{align*}
It follows that there exists some~$\gamma_{k+1} \in \partial\hat{g}(s_{k+1})$ such that
\begin{equation*}
  0 = \gamma_{k+1} + \alpha r_{k+1} - (1-\alpha)s_k + s_{k+1} - \alpha c + u_k .
\end{equation*}
It follows then that
\begin{equation} \label{eq:s_update}
\begin{aligned}
  s_{k+1} & =  - \alpha r_{k+1} + (1-\alpha)s_k + \alpha c - u_k - \gamma_{k+1} \\
  & = s_k - (1-\alpha)u_k + \alpha \beta_{k+1} - \gamma_{k+1} ,
\end{aligned}
\end{equation}
where the second equality follows by substituting in \eqref{eq:r_update}. 
Combining \eqref{eq:r_update} and \eqref{eq:s_update} to simplify the~$u$ update, we have
\begin{equation} \label{eq:u_update}
\begin{aligned}
  u_{k+1} & = u_k + \alpha r_{k+1} - (1-\alpha)s_k + s_{k+1} - \alpha c \\
  & = - \gamma_{k+1} .
\end{aligned}
\end{equation}
Together, \eqref{eq:s_update} and \eqref{eq:u_update} confirm the relation in \eqref{eq:admm_dynamical_system_update_output1}. 
\end{proof}

\begin{corollary} \label{cor:dynamical_system_output}
  Define the sequences~$(\beta_k)$ and~$(\gamma_k)$ as in \propref{prop:dynamical_system_update}. 
  Define the sequences~$(w_k^1)$ and~$(w_k^2)$ via
  \begin{equation*}
    w_k^1= \begin{bmatrix} r_{k+1}-c \\ \beta_{k+1} \end{bmatrix} \quad\quad w_k^2 = \begin{bmatrix} s_{k+1} \\ \gamma_{k+1} \end{bmatrix} .
  \end{equation*}
  Then the sequences~$(\xi_k)$,~$(\nu_k)$,~$(w_k^1)$, and~$(w_k^2)$ satisfy \eqref{eq:admm_dynamical_system_update_output2} and \eqref{eq:admm_dynamical_system_update_output3} with the matrices
  \begin{equation} \label{eqref:admm_matrices_output}
    \begin{gathered}
      \hat{C}^1  = \begin{bmatrix} -1 & -1 \\ 0 & 0 \end{bmatrix} \quad
      \hat{D}^1  = \begin{bmatrix} -1 & 0 \\ 1 & 0 \end{bmatrix} \\
      \hat{C}^2  = \begin{bmatrix} 1 & \alpha-1 \\ 0 & 0 \end{bmatrix} \quad
      \hat{D}^2  = \begin{bmatrix} \alpha & -1 \\ 0 & 1 \end{bmatrix}  .
    \end{gathered}
  \end{equation}
\end{corollary}

\section{Convergence Rates from Semidefinite Programming}
\label{sec:convergence_rates_through_semidefinite_programming}

Now, in \thref{th:lmi_rate}, we make use of the perspective developed in \secref{sec:admm_as_dynamical_system} to obtain convergence rates for \algoref{alg:overrelaxed_admm}. 
This is essentially the same as the main result of \citet{lessard2014analysis}, and we include it because it is simple and self-contained. 

\begin{theorem} \label{th:lmi_rate}
  Suppose that \assref{ass:conditions} holds. 
  Let the sequences~$(x_k)$,~$(z_k)$, and~$(u_k)$ be generated by running \algoref{alg:overrelaxed_admm} with step size~$\rho=(\hat{m}\hat{L})^{\frac12}\rho_0$ and with over-relaxation parameter~$\alpha$. 
  Suppose that~$(x_*,z_*,u_*)$ is a fixed point of \algoref{alg:overrelaxed_admm}, and define
  \begin{equation*}
    \varphi_k = \begin{bmatrix} z_k \\ u_k \end{bmatrix} \quad\quad  \varphi_* = \begin{bmatrix} z_* \\ u_* \end{bmatrix} .
  \end{equation*}
  Fix~$0<\tau<1$, and suppose that there exist a~$2\times2$ positive definite matrix~$P \succ 0$ and nonnegative constants~$\lambda^1,\lambda^2 \ge 0$ such that the~$4\times4$ linear matrix inequality
\begin{equation} \label{eq:lmi}
\begin{aligned} 
  0 \succeq & \begin{bmatrix} \hat{A}^{\top}P\hat{A} - \tau^2P & \hat{A}^{\top}P\hat{B} \\ \hat{B}^{\top}P\hat{A} & \hat{B}^{\top}P\hat{B} \end{bmatrix} \\
  & +  \begin{bmatrix} \hat{C}^{1} & \hat{D}^{1} \\ \hat{C}^{2} & \hat{D}^{2} \end{bmatrix}^{\top} \begin{bmatrix} \lambda^1 M^{1} & 0 \\ 0 & \lambda^2 M^{2} \end{bmatrix} \begin{bmatrix} \hat{C}^{1} & \hat{D}^{1} \\ \hat{C}^{2} & \hat{D}^{2} \end{bmatrix} 
\end{aligned}
\end{equation}
is satisfied, where~$\hat{A}$ and~$\hat{B}$ are defined in \eqref{eq:admm_matrices_update}, where~$\hat{C}^1$,~$\hat{D}^1$,~$\hat{C}^2$, and~$\hat{D}^2$ are defined in \eqref{eqref:admm_matrices_output}, and where~$M^1$ and~$M^2$ are given by
\begin{equation*}
  M^1  = \begin{bmatrix} -2\rho_0^{-2}  & \rho_0^{-1}(\kappa^{-\frac{1}{2}} + \kappa^{\frac{1}{2}}) \\ \rho_0^{-1}(\kappa^{-\frac{1}{2}} + \kappa^{\frac{1}{2}}) & - 2 \end{bmatrix} 
\end{equation*}
\begin{equation*}
  M^2  = \begin{bmatrix} 0 & 1 \\ 1 & 0 \end{bmatrix}  . 
\end{equation*}
Then for all~$k \ge 0$, we have
\begin{equation*}
  \|\varphi_k-\varphi_*\| \le \kappa_B \sqrt{\kappa_P} \|\varphi_0-\varphi_*\| \tau^k .
\end{equation*}
\end{theorem}
\begin{proof}
  Define~$r_k$,~$s_k$,~$\beta_k$,~$\gamma_k$,~$\xi_k$,~$\nu_k$,~$w_k^1$, and~$w_k^2$ as before. 
  Choose~$r_*=Ax_*$,~$s_*=Bz_*$, and
\begin{equation*}
w_*^1=\begin{bmatrix}r_*-c \\ \beta_* \end{bmatrix}
\quad
w_*^2=\begin{bmatrix}s_*\\ \gamma_*\end{bmatrix}
\quad
\xi_*=\begin{bmatrix} s_* \\ u_* \end{bmatrix}
\quad
\nu_*=\begin{bmatrix} \beta_* \\ \gamma_* \end{bmatrix}
\end{equation*}
such that~$(\xi_*,\nu_*,w_*^1,w_*^2)$ is a fixed point of the dynamics of \eqref{eq:admm_dynamical_system_update_output} and satisfying~$\beta_* = \nabla \hat{f}(r_*)$,~$\gamma_* \in \partial \hat{g}(s_*)$. 
Now, consider the Kronecker product of the right hand side of \eqref{eq:lmi} and~$I_r$. 
Multiplying this on the left and on the right by~$\begin{bmatrix} (\xi_j-\xi_*)^{\top} & (\nu_j-\nu_*)^{\top} \end{bmatrix}$ and its transpose, respectively, we find
\begin{equation} \label{eq:lmi_proof}
\begin{aligned}
  0 & \ge (\xi_{j+1}-\xi_*)^{\top} P (\xi_{j+1}-\xi_*) \\
  & \quad - \tau^2 (\xi_j-\xi_*)^{\top} P (\xi_j-\xi_*) \\
  & \quad + \lambda^1 (w_j^1-w_*^1)^{\top}M^1(w_j^1-w_*^1) \\
  & \quad + \lambda^2 (w_j^2-w_*^2)^{\top}M^2(w_j^2-w_*^2) .
\end{aligned}
\end{equation}
\lemref{lem:sec_iqc} and \eqref{eq:fhat_ghata} show that the third term on the right hand side of \eqref{eq:lmi_proof} is nonnegative. 
\lemref{lem:sec_iqc2} and \eqref{eq:fhat_ghatb} show that the fourth term on the right hand side of \eqref{eq:lmi_proof} is nonnegative. 
It follows that
\begin{equation*}
  (\xi_{j+1}-\xi_*)^{\top} P (\xi_{j+1}-\xi_*) \le \tau^2 (\xi_j-\xi_*)^{\top} P (\xi_j-\xi_*) .
\end{equation*}
Inducting from~$j=0$ to~$k-1$, we see that
\begin{equation*}
  (\xi_k-\xi_*)^{\top} P (\xi_k-\xi_*) \le \tau^{2k} (\xi_0-\xi_*)^{\top} P (\xi_0-\xi_*) ,
\end{equation*}
for all~$k$. 
It follows that
\begin{equation*}
  \|\xi_k-\xi_*\| \le  \sqrt{\kappa_P} \|\xi_0-\xi_*\| \tau^k .
\end{equation*}
The conclusion follows.
\end{proof}

For fixed values of~$\alpha$,~$\rho_0$,~$\hat{m}$,~$\hat{L}$, and~$\tau$, the feasibility of \eqref{eq:lmi} is a semidefinite program with variables~$P$,~$\lambda^1$, and~$\lambda^2$. 
We perform a binary search over~$\tau$ to find the minimal rate~$\tau$ such that the linear matrix inequality in \eqref{eq:lmi} is satisfied. 
The results are shown in \figref{fig:numerical_rates} for a wide range of condition numbers~$\kappa$, for~$\alpha=1.5$, and for several choices of~$\rho_0$. 
In \figref{fig:numerical_iters}, we plot the values~$-1/\log\tau$ to show the number of iterations required to achieve a desired accuracy. 

\begin{figure}[ht] 
\centering
\includegraphics{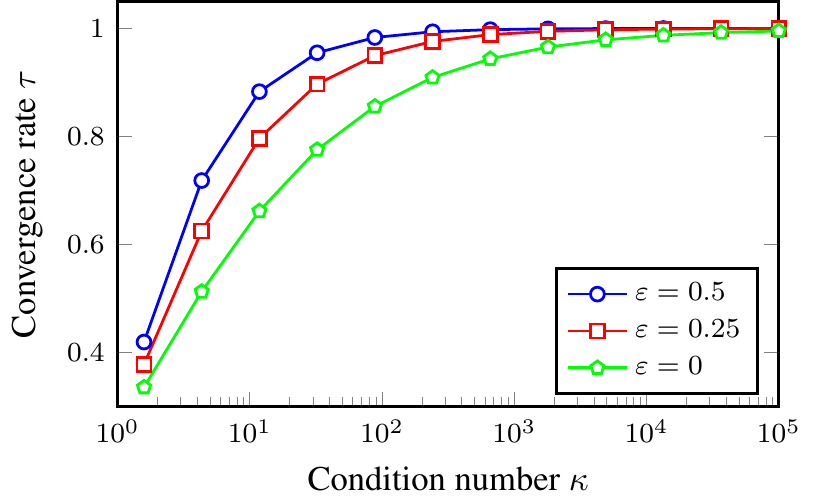}
\caption{For~$\alpha=1.5$ and for several choices of~$\epsilon$ in~$\rho_0=\kappa^{\epsilon}$, we plot the minimal rate~$\tau$ for which the linear matrix inequality in \eqref{eq:lmi} is satisfied as a function of~$\kappa$.}
\label{fig:numerical_rates}
\end{figure}

\begin{figure}[ht] 
\centering
\includegraphics{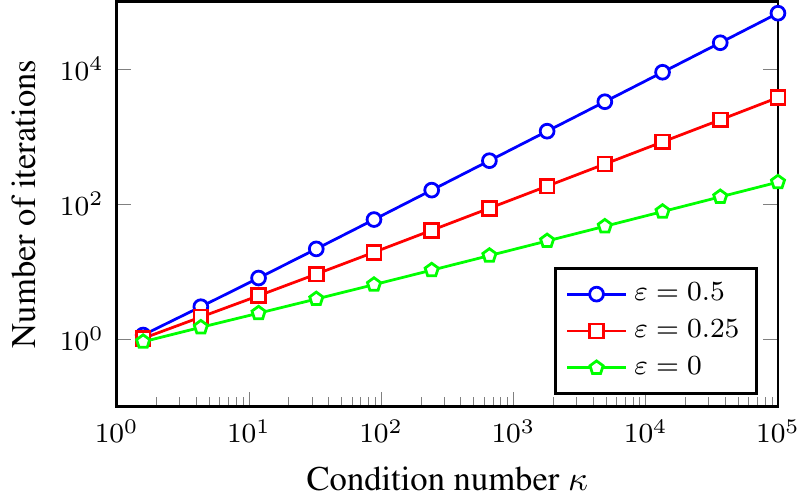}
\caption{For~$\alpha=1.5$ and for several choices of~$\epsilon$ in~$\rho_0=\kappa^{\epsilon}$, we compute the minimal rate~$\tau$ such that the linear matrix inequality in \eqref{eq:lmi} is satisfied, and we plot~$-1/\log\tau$ as a function of~$\kappa$.}
\label{fig:numerical_iters}
\end{figure}

Note that when~$\rho_0=\kappa^{\epsilon}$, the matrix~$M^1$ is given by
\begin{equation*}
  M^1 = \begin{bmatrix} -2\kappa^{-2\epsilon} & \kappa^{-\frac{1}{2}-\epsilon}+\kappa^{\frac{1}{2}-\epsilon} \\ \kappa^{-\frac{1}{2}-\epsilon}+\kappa^{\frac{1}{2}-\epsilon} & -2 \end{bmatrix} ,
\end{equation*}
and so the linear matrix inequality in \eqref{eq:lmi} depends only on~$\kappa$ and not on~$\hat{m}$ and~$\hat{L}$. 
Therefore, we will consider step sizes of this form (recall from \eqref{eq:normalization_defs} that~$\rho=(\hat{m}\hat{L})^{\frac12}\rho_0$). 
The choice~$\epsilon=0$ is common in the literature \citep{giselsson2014diagonal}, but requires the user to know the strong-convexity parameter~$\hat{m}$. 
We also consider the choice~$\epsilon=0.5$, which produces worse guarantees, but does not require knowledge of~$\hat{m}$.

One weakness of \thref{th:lmi_rate} is the fact that the rate we produce is not given as a function of~$\kappa$. 
To use \thref{th:lmi_rate} as stated, we first specify the condition number (for example,~$\kappa=1000$). 
Then we search for the minimal~$\tau$ such that \eqref{eq:lmi} is feasible. 
This produces an upper bound on the convergence rate of \algoref{alg:overrelaxed_admm} (for example,~$\tau=0.9$). 
To remedy this problem, in \secref{sec:specific_rates}, we demonstrate how \thref{th:lmi_rate} can be used to obtain the convergence rate of \algoref{alg:overrelaxed_admm} as a symbolic function of the step size~$\rho$ and the over-relaxation parameter~$\alpha$.

\section{Symbolic Rates for Various~$\rho$ and~$\alpha$} \label{sec:specific_rates}

In \secref{sec:convergence_rates_through_semidefinite_programming}, we demonstrated how to use semidefinite programming to produce numerical convergence rates. 
That is, given a choice of algorithm parameters and the condition number~$\kappa$, we could determine the convergence rate of \algoref{alg:overrelaxed_admm}. 
In this section, we show how \thref{th:lmi_rate} can be used to prove symbolic convergence rates. 
That is, we describe the convergence rate of \algoref{alg:overrelaxed_admm} as a function of~$\rho$,~$\alpha$, and~$\kappa$. 
In \thref{th:analytic_rate}, we prove the linear convergence of \algoref{alg:overrelaxed_admm} for all choices~$\alpha \in (0,2)$ and~$\rho=(\hat{m}\hat{L})^{\frac{1}{2}}\kappa^{\epsilon}$, with~$\epsilon \in (-\infty,\infty)$. 
This result generalizes a number of results in the literature. 
As two examples, \citet{giselsson2014diagonal} consider the case~$\epsilon=0$ and \citet{deng2012global} consider the case~$\alpha=1$ and~$\epsilon=0.5$. 

The rate given in \thref{th:analytic_rate} is loose by a factor of four relative to the lower bound given in \thref{th:lower_bound}. 
However, weakening the rate by a constant factor eases the proof by making it easier to find a certificate for use in \eqref{eq:lmi}.

\begin{theorem} \label{th:analytic_rate}
  Suppose that \assref{ass:conditions} holds. 
  Let the sequences~$(x_k)$,~$(z_k)$, and~$(u_k)$ be generated by running \algoref{alg:overrelaxed_admm} with parameter~$\alpha \in (0,2)$ and with step size~$\rho = (\hat{m}\hat{L})^{\frac{1}{2}}\kappa^{\epsilon}$, where~$\epsilon \in (-\infty,\infty)$. 
  Define~$x_*$,~$z_*$,~$u_*$,~$\varphi_k$, and~$\varphi_*$ as in \thref{th:lmi_rate}. 
  Then for all sufficiently large~$\kappa$, we have
\begin{equation*}
  \|\varphi_k-\varphi_*\| \le C \|\varphi_0-\varphi_*\| \left( 1 - \frac{\alpha}{2\kappa^{0.5 + |\epsilon|}} \right)^k ,
\end{equation*}
where
\begin{equation*}
  C = \kappa_B \sqrt{\max\left\{ \tfrac{\alpha}{2-\alpha}, \tfrac{2-\alpha}{\alpha} \right\}} .
\end{equation*}
\end{theorem}
\begin{proof}
  We claim that for all sufficiently large~$\kappa$, the linear matrix inequality in \eqref{eq:lmi} is satisfied with the rate~$\tau=1-\frac{\alpha}{2\kappa^{0.5+|\epsilon|}}$ and with certificate
\begin{equation*}
  \lambda^1=\alpha \kappa^{\epsilon-0.5}
  \quad 
  \lambda^2=\alpha 
  \quad 
  P = \begin{bmatrix} 1 & \alpha - 1 \\ \alpha - 1 & 1 \end{bmatrix}  .
\end{equation*}
The matrix on the right hand side of \eqref{eq:lmi} can be expressed as~$-\frac{1}{4}\alpha\kappa^{-2} M$, where~$M$ is a symmetric~$4 \times 4$ matrix whose last row and column consist of zeros. 
We wish to prove that~$M$ is positive semidefinite for all sufficiently large~$\kappa$. 
To do so, we consider the cases~$\epsilon \ge 0$ and~$\epsilon < 0$ separately, though the two cases will be nearly identical. 
First suppose that~$\epsilon \ge 0$. 
In this case, the nonzero entries of~$M$ are specified by
\begin{align*}
  M_{11} & = \alpha \kappa^{1-2\epsilon} + 4\kappa^{\frac32-\epsilon} \\
  M_{12} & = \alpha^2\kappa^{1-2\epsilon} - \alpha\kappa^{1-2\epsilon} + 12\kappa^{\frac32-\epsilon}-4\alpha\kappa^{\frac32-\epsilon} \\
  M_{13} & = 4\kappa + 8\kappa^{\frac32-\epsilon} \\
  M_{22} & = 8\kappa^2-4\alpha\kappa^2+\alpha\kappa^{1-2\epsilon} + 4\kappa^{\frac32-\epsilon} \\
  M_{23} & = 4\kappa+8\kappa^2-4\alpha\kappa^2+8\kappa^{\frac32-\epsilon} \\
  M_{33} & = 8\kappa+8\kappa^2 - 4\alpha\kappa^2 + 8\kappa^{\frac32-\epsilon}+8\kappa^{\frac32+\epsilon} .
\end{align*}
We show that each of the first three leading principal minors of~$M$ is positive for sufficiently large~$\kappa$. 
To understand the behavior of the leading principal minors, it suffices to look at their leading terms. 
For large~$\kappa$, the first leading principal minor (which is simple~$M_{11}$) is dominated by the term~$4\kappa^{\frac32-\epsilon}$, which is positive.
Similarly, the second leading principal minor is dominated by the term~$16(2-\alpha)\kappa^{\frac72-\epsilon}$, which is positive. 
When~$\epsilon > 0$, the third leading principal minor is dominated by the term~$128(2-\alpha)\kappa^5$, which is positive. 
When~$\epsilon=0$, the third leading principal minor is dominated by the term~$64\alpha(2-\alpha)^2\kappa^5$, which is positive. 
Since these leading coefficients are all positive, it follows that for all sufficiently large~$\kappa$, the matrix~$M$ is positive semidefinite. 

Now suppose that~$\epsilon < 0$. 
In this case, the nonzero entries of~$M$ are specified by
\begin{align*}
  M_{11} & = 8\kappa^{\frac32-\epsilon}-4\kappa^{\frac32+\epsilon}+\alpha\kappa^{1+2\epsilon} \\
  M_{12} & = 8\kappa^{\frac32-\epsilon}+4\kappa^{\frac32+\epsilon}-4\alpha\kappa^{\frac32+\epsilon}-\alpha\kappa^{1+2\epsilon}+\alpha^2\kappa^{1+2\epsilon} \\ 
  M_{13} & = 4\kappa+8\kappa^{\frac32-\epsilon} \\
  M_{22} & = 8\kappa^2-4\alpha\kappa^2+8\kappa^{\frac32-\epsilon}-4\kappa^{\frac32+\epsilon}+\alpha\kappa^{1+2\epsilon} \\
  M_{23} & = 4\kappa+8\kappa^2-4\alpha\kappa^2+8\kappa^{\frac32-\epsilon} \\
  M_{33} & = 8\kappa+8\kappa^2-4\alpha\kappa^2+8\kappa^{\frac32-\epsilon}+8\kappa^{\frac32+\epsilon} .
\end{align*}
As before, we show that each of the first three leading principal minors of~$M$ is positive. 
For large~$\kappa$, the first leading principal minor (which is simple~$M_{11}$) is dominated by the term~$8\kappa^{\frac32-\epsilon}$, which is positive.
Similarly, the second leading principal minor is dominated by the term~$32(2-\alpha)\kappa^{\frac72-\epsilon}$, which is positive. 
The third leading principal minor is dominated by the term~$128(2-\alpha)\kappa^5$, which is positive. 
Since these leading coefficients are all positive, it follows that for all sufficiently large~$\kappa$, the matrix~$M$ is positive semidefinite. 

The result now follows from \thref{th:lmi_rate} by noting that~$P$ has eigenvalues~$\alpha$ and~$2-\alpha$. 
\end{proof}
Note that since the matrix~$P$ doesn't depend on~$\rho$, the proof holds even when the step size changes at each iteration. 

\section{Lower Bounds} \label{sec:lower_bounds}

In this section, we probe the tightness of the upper bounds on the convergence rate of \algoref{alg:overrelaxed_admm} given by \thref{th:lmi_rate}. 
The construction of the lower bound in this section is similar to a construction given in \citet{ghadimi2014optimal}. 

Let~$Q$ be a~$d$-dimensional symmetric positive-definite matrix whose largest and smallest eigenvalues are~$L$ and~$m$ respectively. 
Let~$f(x)=\frac{1}{2}x^{\top}Qx$ be a quadratic and let~$g(z)=\frac{\delta}{2}\|z\|^2$ for some~$\delta \ge 0$. 
Let~$A=I_d$,~$B=-I_d$, and~$c=0$. 
With these definitions, the optimization problem in \eqref{eq:minimization_problem} is solved by~$x=z=0$. 
The updates for \algoref{alg:overrelaxed_admm} are given by
\begin{subequations} \label{eq:lower_bound_updates}
\begin{align}
  x_{k+1} & = \rho(Q+\rho I)^{-1}(z_k-u_k) \label{eq:lower_bound_updates_a} \\
  z_{k+1} & = \frac{\rho}{\delta + \rho}(\alpha x_{k+1} + (1-\alpha)z_k + u_k) \label{eq:lower_bound_updates_b} \\
  u_{k+1} & = u_k + \alpha x_{k+1} + (1-\alpha)z_k - z_{k+1} . \label{eq:lower_bound_updates_c}
\end{align}
\end{subequations}
Solving for~$z_k$ in \eqref{eq:lower_bound_updates_b} and substituting the result into \eqref{eq:lower_bound_updates_c} gives~$u_{k+1} = \frac{\delta}{\rho}z_{k+1}$. 
Then eliminating~$x_{k+1}$ and~$u_k$ from \eqref{eq:lower_bound_updates_b} using \eqref{eq:lower_bound_updates_a} and the fact that~$u_k=\frac{\delta}{\rho}{z_k}$ allows us to express the update rule purely in terms of~$z$ as
\begin{equation*}
  z_{k+1} = \underbrace{\left( \frac{\alpha\rho(\rho - \delta)}{\rho+\delta}(Q + \rho I)^{-1} + \frac{\rho - \alpha\rho + \delta}{\rho+\delta}I \right)}_T z_k .
\end{equation*}
Note that the eigenvalues of~$T$ are given by
\begin{equation} \label{eq:T_eigs}
  1 - \frac{\alpha\rho(\lambda+\delta)}{(\rho+\delta)(\lambda+\rho)} ,
\end{equation}
where~$\lambda$ is an eigenvalue of~$Q$. 
We will use this setup to construct a lower bound on the worst-case convergence rate of \algoref{alg:overrelaxed_admm} in \thref{th:lower_bound}. 

\begin{theorem} \label{th:lower_bound}
  Suppose that \assref{ass:conditions} holds. 
  The worst-case convergence rate of \algoref{alg:overrelaxed_admm}, when run with step size~$\rho=(\hat{m}\hat{L})^{\frac12}\kappa^{\epsilon}$ and over-relaxation parameter~$\alpha$, is lower-bounded by
  \begin{equation} \label{eq:lower_bound_worst_case}
    1 - \frac{2\alpha}{1 + \kappa^{0.5 + |\epsilon|}} .
  \end{equation}
\end{theorem}
\begin{proof}
  First consider the case~$\epsilon \ge 0$. 
  Choosing~$\delta = 0$ and~$\lambda=m$, from \eqref{eq:T_eigs}, we see that~$T$ has eigenvalue
  \begin{equation} \label{eq:lower_bound_low_epsilon}
    1 - \frac{\alpha}{1 + \kappa^{0.5+\epsilon}} .
  \end{equation}
  When initialized with~$z$ as the eigenvector corresponding to this eigenvalue, \algoref{alg:overrelaxed_admm} will converge linearly with rate given exactly by \eqref{eq:lower_bound_low_epsilon}, which is lower bounded by the expression in \eqref{eq:lower_bound_worst_case} when~$\epsilon \ge 0$. 

  Now suppose that~$\epsilon < 0$. 
  Choosing~$\delta = L$ and~$\lambda = L$, after multiplying the numerator and denominator of~\eqref{eq:T_eigs} by~$\kappa^{0.5-\epsilon}$, we see that~$T$ has eigenvalue
  \begin{equation} \label{eq:lower_bound_higher_epsilon}
    1 - \frac{2\alpha}{(1+\kappa^{0.5-\epsilon})(\kappa^{-0.5+\epsilon}+1)} \ge 1 - \frac{2\alpha}{1+\kappa^{0.5-\epsilon}} .
  \end{equation}
  When initialized with~$z$ as the eigenvector corresponding to this eigenvalue, \algoref{alg:overrelaxed_admm} will converge linearly with rate given exactly by the left hand side of \eqref{eq:lower_bound_higher_epsilon}, which is lower bounded by the expression in \eqref{eq:lower_bound_worst_case} when~$\epsilon < 0$. 
\end{proof}

\figref{fig:upper_lower_bounds} compares the lower bounds given by \eqref{eq:lower_bound_low_epsilon} with the upper bounds given by \thref{th:lmi_rate} for~$\alpha=1.5$ and for several choices of~$\rho=(\hat{m}\hat{L})^{\frac12}\kappa^{\epsilon}$ satisfying~$\epsilon \ge 0$. 
The upper and lower bounds agree visually on the range of choices~$\epsilon$ depicted, demonstrating the practical tightness of the upper bounds given by \thref{th:lmi_rate} for a large range of choices of parameter values. 

\begin{figure}[ht] 
\centering
\includegraphics{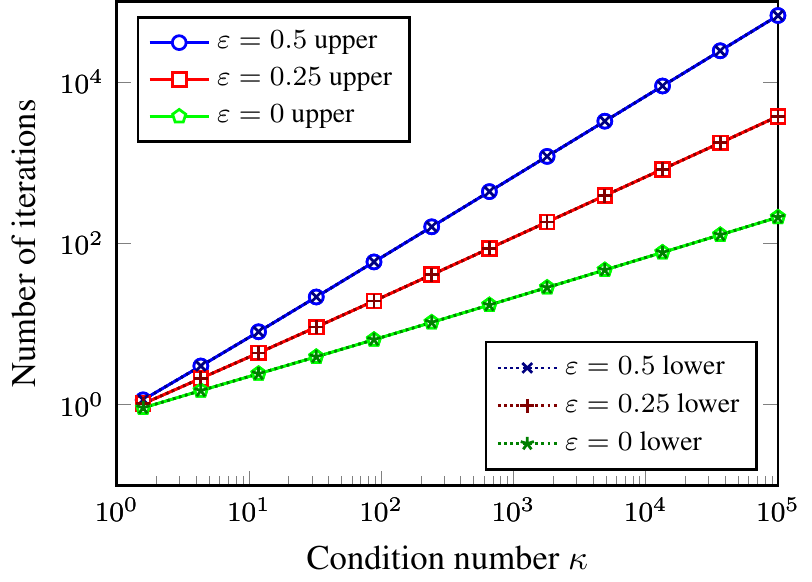}
\caption{For~$\alpha=1.5$ and for several choices~$\epsilon$ in~$\rho_0=\kappa^{\epsilon}$, we plot~$-1/\log\tau$ as a function of~$\kappa$, both for the lower bound on~$\tau$ given by \eqref{eq:lower_bound_low_epsilon} and the upper bound on~$\tau$ given by \thref{th:lmi_rate}. 
For each choice of~$\epsilon$ in~$\{0.5,0.25,0\}$, the lower and upper bounds agree visually. 
This agreement demonstrates the practical tightness of the upper bounds given by \thref{th:lmi_rate} for a large range of choices of parameter values.}
\label{fig:upper_lower_bounds}
\end{figure}

\section{Related Work}
\label{sec:related_work}

Several recent papers have studied the linear convergence of \algoref{alg:admm} but do not extend to \algoref{alg:overrelaxed_admm}. 
\citet{deng2012global} prove a linear rate of convergence for ADMM in the strongly convex case. 
\citet{iutzeler2014linear} prove the linear convergence of a specialization of ADMM to a class of distributed optimization problems under a local strong-convexity condition. 
\citet{hong2012linear} prove the linear convergence of a generalization of ADMM to a multiterm objective in the setting where each term can be decomposed as a strictly convex function and a polyhedral function. 
In particular, this result does not require strong convexity.

More generally, there are a number of results for operator splitting methods in the literature. 
\citet{lions1979splitting} and \citet{eckstein1998operator} analyze the convergence of several operator splitting schemes. 
More recently, \citet{patrinos2014douglas,patrinos2014forward} prove the equivalence of forward-backward splitting and Douglas--Rachford splitting with a scaled version of the gradient method applied to unconstrained nonconvex surrogate functions (called the forward-backward envelope and the Douglas--Rachford envelope respectively). 
\citet{goldstein2012fast} propose an accelerated version of ADMM in the spirit of Nesterov, and prove a~$O(1/k^2)$ convergence rate in the case where~$f$ and~$g$ are both strongly convex and~$g$ is quadratic. 

The theory of over-relaxed ADMM is more limited. 
\citet{eckstein1992douglas} prove the convergence of over-relaxed ADMM but do not give a rate. 
More recently, \citet{davis2014convergenceA,davis2014convergenceB} analyze the convergence rates of ADMM in a variety of settings. 
\citet{giselsson2014diagonal} prove the linear convergence of Douglas--Rachford splitting in the strongly-convex setting. 
They use the fact that ADMM is Douglas--Rachford splitting applied to the dual problem \citep{eckstein1992douglas} to derive a linear convergence rate for over-relaxed ADMM with a specific choice of step size~$\rho$. 
\citet{eckstein1994parallel} gives convergence results for several specializations of ADMM, and found that over-relaxation with~$\alpha=1.5$ empirically sped up convergence. 
\citet{ghadimi2014optimal} give some guidance on tuning over-relaxed ADMM in the quadratic case. 

Unlike prior work, our framework requires no assumptions on the parameter choices in \algoref{alg:overrelaxed_admm}. 
For example, \thref{th:lmi_rate} certifies the linear convergence of \algoref{alg:overrelaxed_admm} even for values~$\alpha > 2$. 
In our framework, certifying a convergence rate for an arbitrary choice of parameters amounts to checking the feasibility of a~$4\times 4$ semidefinite program, which is essentially instantaneous, as opposed to formulating a proof.

\section{Selecting Algorithm Parameters} \label{sec:selecting_hyperparameters}

In this section, we show how to use the results of \secref{sec:convergence_rates_through_semidefinite_programming} to select the parameters~$\alpha$ and~$\rho$ in \algoref{alg:overrelaxed_admm} and we show the effect on a numerical example. 

Recall that given a choice of parameters~$\alpha$ and~$\rho$ and given the condition number~$\kappa$, \thref{th:lmi_rate} gives an upper bound on the convergence rate of \algoref{alg:overrelaxed_admm}. 
Therefore, one approach to parameter selection is to do a grid search over the space of parameters for the choice that minimizes the upper bound provided by \thref{th:lmi_rate}. 
We demonstrate this approach numerically for a distributed Lasso problem, but first we demonstrate that the usual range of~$(0,2)$ for the over-relaxation parameter~$\alpha$ is too limited, that more choices of~$\alpha$ lead to linear convergence. 
In \figref{fig:alpha_interval}, we plot the largest value of~$\alpha$ found through binary search such that \eqref{eq:lmi} is satisfied for some~$\tau < 1$ as a function of~$\kappa$. 
Proof techniques in prior work do not extend as easily to values of~$\alpha > 2$. 
In our framework, we simply change some constants in a small semidefinite program. 

\begin{figure}[ht] 
\centering
\includegraphics{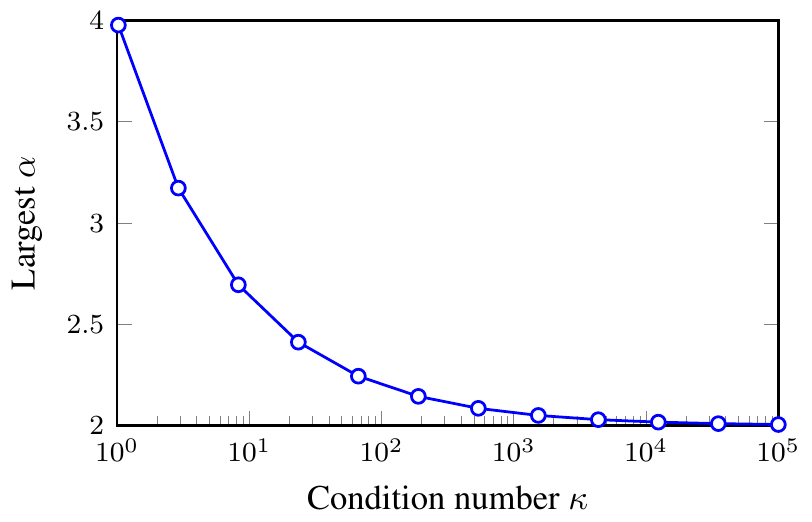}
\caption{As a function of~$\kappa$, we plot the largest value of~$\alpha$ such that \eqref{eq:lmi} is satisfied for some~$\tau<1$. 
In this figure, we set~$\epsilon=0$ in~$\rho_0=\kappa^{\epsilon}$.}
\label{fig:alpha_interval}
\end{figure}

\subsection{Distributed Lasso} \label{sec:distributed_lasso}

Following \citet{deng2012global}, we give a numerical demonstration with a distributed Lasso problem of the form
\begin{align*}
  \text{minimize} & \quad \sum_{i=1}^N \frac{1}{2\mu} \| A_ix_i - b_i\|^2 + \|z\|_1 \\
    \text{subject to} & \quad x_i-z=0 \quad \text{for all} \quad i=1,\ldots,N .
\end{align*}
Each~$A_i$ is a tall matrix with full column rank, and so the first term in the objective will be strongly convex and its gradient will be Lipschitz continuous. 
As in \citet{deng2012global}, we choose~$N=5$ and~$\mu=0.1$. 
Each~$A_i$ is generated by populating a~$600 \times 500$ matrix with independent standard normal entries and normalizing the columns. 
We generate each~$b_i$ via~$b_i=A_ix^0+\epsilon_i$, where~$x^0$ is a sparse~$500$-dimensional vector with~$250$ independent standard normal entries, and~$\epsilon_i \sim \mathcal N(0,10^{-3}I)$. 

In \figref{fig:hyperparameter_selection}, we compute the upper bounds on the convergence rate given by \thref{th:lmi_rate} for a grid of values of~$\alpha$ and~$\rho$. 
Each line corresponds to a fixed choice of~$\alpha$, and we plot only a subset of the values of~$\alpha$ to keep the plot manageable. 
We omit points corresponding to parameter values for which the linear matrix inequality in \eqref{eq:lmi} was not feasible for any value of~$\tau < 1$.

\begin{figure}[ht] 
\centering
\includegraphics{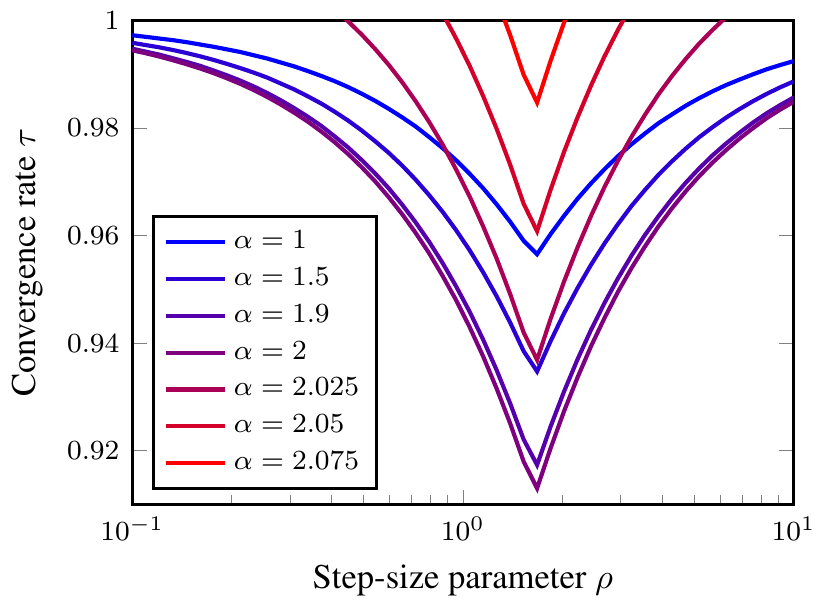}
\caption{We compute the upper bounds on the convergence rate given by \thref{th:lmi_rate} for eighty-five values of~$\alpha$ evenly spaced between~$0.1$ and~$2.2$ and fifty values of~$\rho$ geometrically spaced between~$0.1$ and~$10$. 
Each line corresponds to a fixed choice of~$\alpha$, and we show only a subset of the values of~$\alpha$ to keep the plot manageable. 
We omit points corresponding to parameter values for which \eqref{eq:lmi} is not feasible for any value of~$\tau < 1$. 
This analysis suggests choosing~$\alpha=2.0$ and~$\rho=1.7$.}
\label{fig:hyperparameter_selection}
\end{figure}

In \figref{fig:distributed_lasso_dists}, we run \algoref{alg:overrelaxed_admm} for the same values of~$\alpha$ and~$\rho$. 
We then plot the number of iterations needed for~$z_k$ to reach within~$10^{-6}$ of a precomputed reference solution. 
We plot lines corresponding to only a subset of the values of~$\alpha$ to keep the plot manageable, and we omit points corresponding to parameter values for which \algoref{alg:overrelaxed_admm} exceeded~$1000$ iterations. 
For the most part, the performance of \algoref{alg:overrelaxed_admm} as a function of~$\rho$ closely tracked the performance predicted by the upper bounds in \figref{fig:hyperparameter_selection}. 
Notably, smaller values of~$\alpha$ seem more robust to poor choices of~$\rho$. 
The parameters suggested by our analysis perform close to the best of any parameter choices. 

\begin{figure}[ht] 
\centering
\includegraphics{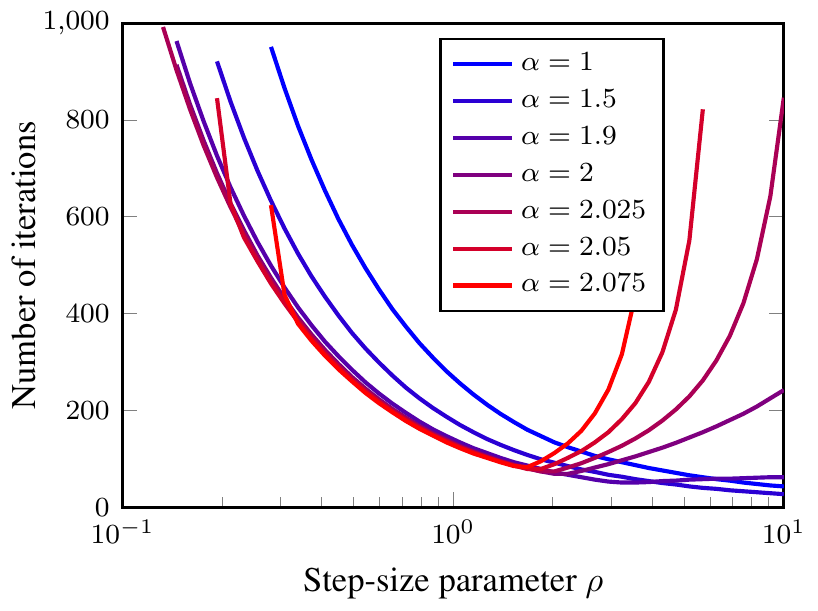}
\caption{We run \algoref{alg:overrelaxed_admm} for eighty-five values of~$\alpha$ evenly spaced between~$0.1$ and~$2.2$ and fifty value of~$\rho$ geometrically spaced between~$0.1$ and~$10$. 
We plot the number of iterations required for~$z_k$ to reach within~$10^{-6}$ of a precomputed reference solution. 
We show only a subset of the values of~$\alpha$ to keep the plot manageable. 
We omit points corresponding to parameter values for which \algoref{alg:overrelaxed_admm} exceeded~$1000$ iterations. } 
\label{fig:distributed_lasso_dists}
\end{figure}

\section{Discussion}
\label{sec:discussion}

We showed that a framework based on semidefinite programming can be used to prove convergence rates for the alternating direction method of multipliers and allows a unified treatment of the algorithm's many variants, which arise through the introduction of additional parameters. 
We showed how to use this framework for establishing convergence rates, as in \thref{th:lmi_rate} and \thref{th:analytic_rate}, and how to use this framework for parameter selection in practice, as in \secref{sec:selecting_hyperparameters}. 
The potential uses are numerous. 
This framework makes it straightforward to propose new algorithmic variants, for example, by introducing new parameters into \algoref{alg:overrelaxed_admm} and using \thref{th:lmi_rate} to see if various settings of these new parameters give rise to improved guarantees. 

In the case that \assref{ass:conditions} does not hold, the most likely cause is that we lack the strong convexity of~$f$. 
One approach to handling this is to run \algoref{alg:overrelaxed_admm} on the modified function~$f(x)+\frac{\delta}{2}\|x\|^2$. 
By completing the square in the~$x$ update, we see that this amounts to an extremely minor algorithmic modification (it only affects the~$x$ update). 

It should be clear that other operator splitting methods such as Douglas--Rachford splitting and forward-backward splitting can be cast in this framework and analyzed using the tools presented here.

\section*{Acknowledgments} 
This research is supported in part by NSF CISE Expeditions award CCF-1139158, LBNL award 7076018, DARPA XData award FA8750-12-2-0331, AFOSR award FA9550-12-1-0339, NASA grant NRA NNX12AM55A, ONR grants N00014-11-1-0688 and N00014-14-1-0024, US ARL and US ARO grant W911NF-11-1-0391, NSF awards CCF-1359814 and CCF-1217058, NSF grant DGE-1106400, a Sloan Research Fellowship, and gifts from Amazon Web Services, Google, SAP, The Thomas and Stacey Siebel Foundation, Adatao, Adobe, Apple, Blue Goji, Bosch, C3Energy, Cisco, Cray, Cloudera, EMC, Ericsson, Facebook, Guavus, Huawei, Informatica, Intel, Microsoft, NetApp, Pivotal, Samsung, Splunk, Virdata, VMware, and Yahoo!. 

\bibliography{refs}
\bibliographystyle{icml2015}

\end{document}